\documentclass[a4paper,reqno]{amsart}
\usepackage{amssymb, amsthm, amsmath}
\usepackage{enumitem}
\usepackage{hyperref}
\usepackage{xparse}
\usepackage{csquotes}

\newtheorem{theorem}{Theorem}[section]

\theoremstyle{definition}

\newtheorem{remark}[theorem]{Remark}

\newcommand{\Addresses}{{
\footnotesize
\bigskip
\footnotesize
		
\textsc{Department of Mathematics and Statistics, Queen's University, Kingston, Ontario, K7L 3N6, Canada}\par\nopagebreak
\textit{E-mail address:}
\texttt{brad.w.rodgers@gmail.com}		
}}

\title[weighted Hilbert inequality]{On the optimal constant in the Montgomery-Vaughan weighted Hilbert inequality}
\author{Brad Rodgers}
\date{}

\begin{document}
	
\begin{abstract}
We give a proof that the optimal constant in the weighted Hilbert inequality is strictly greater than $\pi$, answering in the negative a question asked by Montgomery and Vaughan. The proof proceeds via an explicit limiting counterexample.
\end{abstract}

\maketitle
\thispagestyle{empty}

\section{Introduction}
\label{sec:intro}

The weighted version of Hilbert's inequality due to Montgomery and Vaughan \cite{MoVa74} states that there is an absolute constant $C$ such that for any $N\geq 2$, any distinct real numbers $\lambda_1,...,\lambda_N$, and any real or complex numbers $w_1,...,w_N$, if $\delta_n = \min_{m\neq n}|\lambda_m-\lambda_n|$ we have
\begin{equation}
\label{eq:infinite_Hilbert}
\Big| \sum_{\substack{1 \leq m,n \leq N \\ m \neq n}}  \frac{w_m \overline{w_n}}{\lambda_m - \lambda_n} \Big| \leq C \sum_{1\leq m \leq N} \frac{|w_m|^2}{\delta_m}.
\end{equation}
In \cite{MoVa74}, Montgomery and Vaughan proved this inequality with $C = 3\pi/2$. 

The classical Hilbert inequality is the special case $\lambda_m = m$ for all $m$. In this case, improving the estimate of Hilbert, Schur \cite{Sch} proved 
\[
\Big| \sum_{\substack{1 \leq m,n \leq N \\ m \neq n}}  \frac{w_m \overline{w_n}}{m - n} \Big| \leq \pi \sum_{1\leq m \leq N} |w_m|^2.
\]
Moreover, if $\pi$ were replaced by any smaller absolute constant this inequality would be false once $N$ grows sufficiently large. Montgomery and Vaughan also proved a variant of \eqref{eq:infinite_Hilbert} in which each $\delta_m$ is replaced by $\delta = \min \delta_m$. In that case as well the inequality can be proved with an absolute constant $\pi$ and again this is best possible.

There has therefore been some interest in obtaining the best possible absolute constant in the inequality \eqref{eq:infinite_Hilbert}. This is given as Problem 21 in the Appendix of \cite{Mo}. In \cite{Pr} Preissmann proved that $C$ may be taken to be $\pi \sqrt{1 + \frac{2}{3}\sqrt{\frac{6}{5}}} \approx 4.1325$, which appears to be the best known upper bound in the literature. Selberg claimed to Montgomery and Vaughan that a constant $C = 3.2$ could be obtained, but it appears there is no record of his argument. (He may have had the reduction in mind which later appeared as \cite[Thm. 1.1]{Ya}; Yangjit shows in that paper that this approach cannot possibly achieve a constant better than $C \approx 3.195$, though it remains an open problem in the literature whether the method can achieve this constant.) Other sources on this topic include \cite{CaLi,Li, MoVaaler99,PrLe}.

In light of the classical Hilbert inequality and related results, it is natural to wonder whether \eqref{eq:infinite_Hilbert} holds with $C = \pi$; already in \cite{MoVa74} this question is posed. The purpose of this note is to show that this is not the case.

\begin{theorem}
\label{thm:main}
If \eqref{eq:infinite_Hilbert} holds for every $N \geq 2$, and all admissible $\lambda$ and $w$, then the constant $C$ must be chosen such that $C > \pi$.
\end{theorem}

The short proof proceeds via a counterexample construction given in the next section. An explicit quantity $\alpha_0 > \pi$ is given such that we must have $C \geq \alpha_0$.

\subsection*{Acknowledgments}
I thank Bahman Gharesifard, Ofir Gorodetsky, and Jeff Lagarias for helpful feedback during the preparation of this manuscript. This research is supported in part by an NSERC Discovery Grant. Some of the work related to this manuscript was carried out during visits to Stanford University in 2024 and the Centre de recherches math\'ematiques (CRM) in Montreal during 2026 Universal Statistics in Number Theory thematic semester and I thank these institutions for their hospitality.

\subsection*{Use of large language models} 
ChatGPT 5.6 Sol has been used in this proof. The numerical experiments described in Remark \ref{rem:numerics} formed the start of this project. These began in 2024 but it was only recently that programs written by ChatGPT 5.6 Sol uncovered unexpected high-dimensional behavior. A further back and forth about these numerics led ultimately to the proof of Theorem 1.1 given in Section \ref{sec:geq_pi}. The proof that appears here is a modification and simplification of a proof suggested by ChatGPT. ChatGPT has also been used to proofread and format, but all writing in this manuscript was done by the author. 

\section{A proof that \texorpdfstring{$C>\pi$}{C > pi} and further commentary}
\label{sec:geq_pi}

\begin{proof}[Proof of Theorem \ref{thm:main}]
Suppose $C$ is a constant such that \eqref{eq:infinite_Hilbert} holds for all $N$, $\lambda$ and $w$. For notational convenience we will make a change of variable $w_m = z_m \sqrt{\delta_m}$ in \eqref{eq:infinite_Hilbert} and moreover relabel the range of summation from $\{1,...,N\}$ to $\{-1,...,M\}$, so that the inequality becomes
\begin{equation}
\label{eq:finite_Hilbert}
\Big| \sum_{\substack{-1 \leq j,k \leq M \\ j \neq k}}  \frac{z_j \overline{z_k} \sqrt{\delta_j \delta_k}}{\lambda_j - \lambda_k} \Big| \leq C \sum_{-1\leq j \leq M} |z_j|^2.
\end{equation}

We consider the case that $\lambda_{-1}, \lambda_0, \lambda_1, \lambda_2,..., \lambda_M$ are given by $-1, 0, \varepsilon, 2\varepsilon, ..., M\varepsilon$ for positive $\varepsilon = \varepsilon(M) \rightarrow 0$ in such a way that $M\varepsilon \rightarrow\infty$. Thus for $\varepsilon < 1$, we have $\delta_{-1} = 1$ and $\delta_j = \varepsilon$ for $j\geq 0$. Specialize to a vector
\[
z_{-1} = z, \quad z_j = \sqrt{\varepsilon} u(j \varepsilon)
\]
for $z \in \mathbb{C}$ and $u\colon \mathbb{R} \rightarrow \mathbb{C}$. We will give explicit $z$ and $u$ below such that the ratio of the sum on the left-hand side of \eqref{eq:finite_Hilbert} and the sum on the right-hand side will tend to a value greater than $\pi$, but it is easier to keep track of computations by using a Hilbert space formulation. 

Consider for the moment $u \in C_c^\infty(0,\infty)$ (with the convention that $u(0) = 0$ for such a $u$). If we suppose \eqref{eq:finite_Hilbert} then we have
\begin{multline*}
\left|\sum_{\substack{0 \leq j, k \leq M \\ j\neq k}} \frac{\sqrt{\varepsilon} u(j\varepsilon) \sqrt{\varepsilon} \overline{u(k\varepsilon)}}{\varepsilon j - \varepsilon k}
 \sqrt{\varepsilon \cdot \varepsilon} + \sum_{0 \leq k \leq M} \frac{z \sqrt{\varepsilon} \overline{u(k\varepsilon)}}{-1 - \varepsilon k} \sqrt{\varepsilon} + \sum_{0 \leq j \leq M} \frac{\sqrt{\varepsilon} u(j\varepsilon) \overline{z}}{\varepsilon j + 1} \sqrt{\varepsilon}\right|\\
 \leq C\Big( |z|^2 + \sum_{0 \leq j \leq M} |\sqrt{\varepsilon} u(j\varepsilon)|^2\Big).
\end{multline*}
Taking $M\rightarrow\infty$ yields\footnote{A careful justification for convergence of the double sum is given in Remark \ref{rem:double_conv} below.}
\begin{multline}
\label{eq:limiting_inequality}
\left|\int_0^\infty PV \int_0^\infty \frac{u(s)\overline{u(t)}}{s-t} \, dt ds - z \int_0^\infty \frac{\overline{u(t)}}{1+t}\, dt + \overline{z} \int_0^\infty \frac{u(s)}{1+s} \, ds\right| \\
 \leq C\Big( |z|^2 + \int_0^\infty |u(s)|^2\, ds \Big),
\end{multline}
where $PV$ indicates the inner integral is a principal value integral.

Now consider the Hilbert space $\mathcal{H} = L^2[0,\infty)\oplus \mathbb{C}$ with inner product
\[
\left\langle \begin{pmatrix} u_1 \\ z_1 \end{pmatrix}, \begin{pmatrix} u_2 \\ z_2 \end{pmatrix} \right\rangle = \int_0^\infty u_1(s) \overline{u_2(s)}\, ds + z_1 \overline{z_2}.
\]
Define the map
\[
W \begin{pmatrix} u \\ z \end{pmatrix} = \begin{pmatrix} v \\ -\int_0^\infty \frac{u(t)}{1+t}\, dt \end{pmatrix}, \quad \textrm{where}\; v(s) = Hu(s) + \frac{z}{1+s},
\]
and where $H\colon L^2[0,\infty) \rightarrow L^2[0,\infty)$ is the Hilbert transform, defined by $Hu(s) = PV \int_0^\infty \frac{u(t)}{s-t}\, dt$ for $u \in C_c^\infty(0,\infty)$. We have that $W$ is a bounded linear operator on $\mathcal{H}$ because the Hilbert transform is $L^2$-bounded and because $1/(1+s)$ is square integrable.

Thus by \eqref{eq:limiting_inequality}
\[
|\langle \psi, W\psi \rangle |\leq C \|\psi\|^2,
\]
for all $\psi$ in the subspace $C_c^\infty(0,\infty)\oplus \mathbb{C}$ of $\mathcal{H}$. By density and the boundedness of $W$ this inequality extends to all $\psi \in \mathcal{H}$.

In particular, one sees that if $W$ has an eigenvalue $\lambda$, then $|\lambda| \leq C$. But we will demonstrate that $W$ has an eigenvalue with modulus greater than $\pi$.

For $t > 0$, define
\[
u_\alpha(t) = \frac{t^{i\alpha}}{1+t},
\] 
for a real parameter $\alpha \neq 0$ to be chosen. Set $u_\alpha(0) = 0$. We have $u_\alpha \in L^2[0,\infty)$ and furthermore we have the calculus identities\footnote{The identity \eqref{eq:calc1} follows directly from an appropriate choice of parameters in the principal value integral \cite[6.2 (8), p.309]{Er}.
The identity \eqref{eq:calc2} follows as a limit of \cite[6.2 (7), p.309]{Er}.
The Fourier analytic approach outlined in Remark \ref{rem:fourier_ansatz} can also furnish an alternative proof.},
\begin{align}
\label{eq:calc1}
PV \int_0^\infty \frac{u_\alpha(t)}{s-t}\, dt &= \Big[- i \pi \frac{\cosh(\pi \alpha)}{\sinh(\pi \alpha)} \Big] u_\alpha(s) + \Big[ \frac{i\pi}{\sinh(\pi \alpha)} \Big]\frac{1}{1+s}, \\
\label{eq:calc2}
\int_0^\infty \frac{u_\alpha(t)}{1+t}\, dt &= \frac{\pi \alpha}{\sinh(\pi \alpha)}.
\end{align}
Set
\[
\psi_\alpha = \begin{pmatrix} u_\alpha \\ -i\pi/\sinh(\pi \alpha) \end{pmatrix}.
\]
We have chosen this value of $z$ in order to cancel the second term in the right-hand side of \eqref{eq:calc1} when $W$ is applied, and we have
\[
W\psi_\alpha = \begin{pmatrix} \Big[ -i \pi \frac{\cosh(\pi \alpha)}{\sinh(\pi \alpha)}\Big] u_\alpha \\ -\pi \alpha / \sinh(\pi \alpha) \end{pmatrix}.
\]

Note that $\psi_\alpha$ will be an eigenvector with eigenvalue $\lambda$ provided that $W \psi_\alpha = \lambda\psi_\alpha$. This happens when
\begin{align*}
-i \pi \frac{\cosh(\pi \alpha)}{\sinh(\pi \alpha)} &= \lambda \\
\alpha &= i\lambda,
\end{align*}
or in other words when
\begin{equation}
\label{eq:evalue_solved}
\alpha = \pi \frac{\cosh(\pi \alpha)}{\sinh(\pi \alpha)}.
\end{equation}
We will find a value of $\alpha_0 \in (\pi,\infty)$ which solves this. Consider
\[
F(\alpha) = \alpha - \pi \frac{\cosh(\pi \alpha)}{\sinh(\pi \alpha)}.
\]
Because $\cosh(x) > \sinh(x)$ for all real $x$, we have
\[
F(\pi) = \pi\Big(1 - \frac{\cosh(\pi^2)}{\sinh(\pi^2)}\Big) < 0
\]
and as $\alpha \rightarrow\infty$,
\[
F(\alpha) = \alpha - \pi + o(1) \rightarrow \infty.
\]
By continuity, $F$ has some zero $\alpha_0 \in (\pi, \infty)$.\footnote{Numerically $\alpha_0 \approx \pi+ 1.68 \times 10^{-8}.$ Note that this is a very crude lower bound for $C$; numerically it appears $C$ is larger still.}  Therefore \eqref{eq:evalue_solved} is satisfied for $\alpha = \alpha_0 > \pi$, and $W$ has an eigenvalue $\lambda = -i\alpha_0$ with $|\lambda| > \pi$.

Hence $\pi < |\lambda| \leq C$ as claimed.
\end{proof}

\begin{remark}
\label{rem:double_conv} 
We have promised a more careful justification of convergence to the double integral in \eqref{eq:limiting_inequality}.

Define
\[
U(s,t) = \begin{cases}
\big(u(s)\overline{u(t)} - u(t) \overline{u(s)}\big)/(s-t) & s\neq t \\
u'(s) \overline{u(s)} - u(s)\overline{u'(s)} & s=t.
\end{cases}
\]
One checks $U \in C_c((0,\infty)^2)$. By symmetrizing, one sees that for sufficiently large $M$, the discrete double sum above \eqref{eq:limiting_inequality} is 
\begin{multline*}
\varepsilon^2 \sum_{\substack{j,k\geq 0 \\ j\neq k}} \frac{u(j\varepsilon) \overline{u(k\varepsilon)}}{j\varepsilon - k\varepsilon} = \frac{1}{2} \varepsilon^2 \sum_{\substack{j,k\geq 0 \\ j\neq k}} \frac{u(j\varepsilon) \overline{u(k\varepsilon)} - u(k\varepsilon) \overline{u(j\varepsilon)}}{j\varepsilon - k\varepsilon}\\
 = \frac{1}{2}\varepsilon^2 \sum_{j,k\geq 0} U(j\varepsilon, k \varepsilon) + O(\varepsilon) \rightarrow \frac{1}{2}\int_0^\infty \int_0^\infty U(s,t)\, ds dt.
\end{multline*}
But using $u(t) = u(s) + O(|t-s|)$ and the compact support of $u$, we have as $\delta \rightarrow 0$,
\[
PV\int_0^\infty \frac{u(t)}{s-t}\, dt = \int_{|s-t|\geq \delta} \frac{u(t)}{s-t}\, dt + O(\delta),
\]
uniformly in $s$. Therefore (again using the compact support of $u$),
\begin{multline*}
\int_0^\infty PV \int_0^\infty \frac{u(s)\overline{u(t)}}{s-t} \, dt ds = \iint_{\substack{|s-t| \geq \delta \\ s,t > 0}}\frac{u(s)\overline{u(t)}}{s-t} \, dt ds + O(\delta) \\
= \frac{1}{2}\iint_{\substack{|s-t| \geq \delta \\ s,t > 0}} U(s,t)\, dtds + O(\delta) \rightarrow \frac{1}{2}\int_0^\infty \int_0^\infty U(s,t)\, ds dt,
\end{multline*}
as $\delta \rightarrow 0$. In the last line we have used symmetrization as before and the fact that $U$ is bounded.
\end{remark}

\begin{remark}
\label{rem:fourier_ansatz}
The ansatz $u_\alpha(t) = t^{i\alpha}/(1+t)$ can be motivated in the following way. In making it, we needed to solve
\[
PV \int_0^\infty \frac{u(t)}{s-t}\, dt + \frac{z}{1+s} = \lambda u(s),
\]
for some values of $\lambda, z$. Scale invariances suggest a change of variable $s = e^x$ and $t = e^y$. After rearrangement the above becomes
\[
PV \int_{-\infty}^\infty \frac{u(e^y) e^{y/2}}{\sinh((x-y)/2)} \, dy + \frac{z}{\cosh(x/2)}  = 2\lambda u(e^x) e^{x/2}.
\]
Setting $w(x) = u(e^x)e^{x/2}$ and formally taking Fourier transforms under the convention $\hat{w}(\xi) = \int w(x) e^{-i 2 \pi x\xi}\, dx$ we obtain\footnote{The Fourier transform formulas for the reciprocal hyperbolic trigonometric functions can be obtained from \cite[3.981.1 and 3.981.3, p. 509]{GrRy}}
\[
\hat{w}(\xi) \cdot \Big[-2\pi i\frac{\sinh(2\pi^2\xi)}{\cosh(2\pi^2 \xi)}\Big] + \frac{2\pi z}{\cosh(2\pi^2 \xi)} = 2 \lambda \hat{w}(\xi).
\]
Solving for $\hat{w}(\xi)$ we find it is the reciprocal of some linear combination of $\sinh(2\pi^2\xi)$ and $\cosh(2\pi^2 \xi)$. In light of the hyperbolic trigonometric addition formula, this suggests that $\hat{w}(\xi)$ is a scalar multiple of $1/\cosh(2\pi^2 (\xi-\beta))$ for some translation $\beta$. Taking the inverse Fourier transform and solving for $u(t)$ suggests it should be a scalar multiple of $t^{i2\pi \beta}/(1+t)$.

As we have noted, the calculus identity for the Hilbert transform of $u_\alpha$ can also be derived using Fourier analysis in this way.

This change of variables can also be equivalently viewed as exploiting a well-known relationship between the Hilbert transform and the Mellin transform (see e.g. \cite[Ch. 5.8]{Ki}).
\end{remark}

\begin{remark}
\label{rem:numerics}
For fixed $N\geq 2$, let $C_N$ be the infimum of all $C$ satisfying \eqref{eq:infinite_Hilbert} for all $\lambda_1,...,\lambda_N$ and $w_1,...,w_N$. It is an exercise to show that there must be some configurations $\lambda_1,...,\lambda_N$ which realize the constant $C_N$ for some input vectors $w$. The proof above began with a numerical study of such extremal configurations. 

For small $N$ (e.g. $N=5$) it can be shown by a somewhat tedious calculation that the extremal configurations of $\lambda$ are those which are equally spaced, e.g. $\lambda_1 = 1,...,\lambda_N = N$. But this pattern breaks down already for $N = 21$. In view of Theorem \ref{thm:main} and Schur's result this must happen for some $N$, but it came as a surprise to the author at first, before Theorem \ref{thm:main} was proved. It was an extrapolation to larger $N$ of extremal patterns for relatively small $N\geq 21$ which led to a first numerical proof of Theorem \ref{thm:main}. This first proof came from explicit $\lambda$ and $w$ for $N = 10000$ and a numerical certificate.

In light of the proof given above, that numerical certificate is perhaps of limited interest. But numerical investigations of extremal patterns of $\lambda$ nonetheless reveal some conjectural patterns which seem curious and suggest the value of the optimal constant $C$ may be around $3.143$. More information about conjectural extremal patterns as well as relevant Python programs are made available at \url{https://github.com/brad-rodgers/extremalconfigs}. These programs are for the purpose of numerical experiments; we caution that they have been generated by ChatGPT 5.6 Sol and have not been independently audited at this time.
\end{remark}

	\Addresses
	
\end{document}